\documentclass{gen-j-l}

\usepackage{amssymb}
\usepackage{amsrefs}
\usepackage{graphicx}
\usepackage{bm}
\usepackage{mathrsfs}
\usepackage{amsmath}
\usepackage{yhmath}
\usepackage{adjustbox}
\usepackage{mathtools} % 放在导言区
\usepackage{tikz}
\usepackage{hyperref}
\usetikzlibrary{calc, angles, quotes}
\usetikzlibrary{patterns}
\usetikzlibrary{intersections}
\newtheorem{theorem}{Theorem}[section]
\newtheorem{lemma}[theorem]{Lemma}

\theoremstyle{definition}

\numberwithin{equation}{section}
\begin{document}

% \title{An Improvement on the Lower Bound for the Area of Star-Shaped Kakeya Sets}
% \title{An Improvement on the Lower Bound for Star-Shaped Kakeya Sets}
\title{An improved lower bound \\for star-shaped Kakeya sets}

%    author one information
\author{Shaoqi Li}
\address{Department of Mathematics, Sun Yat-sen University, Guangzhou,
 Guangdong 510275, P.R. China}
% \curraddr{}
\email{lishq76@mail2.sysu.edu.cn}
% \thanks{}
% \author{}
% \address{}
% \curraddr{}
% \email{}
% \thanks{}

%    \subjclass is required by all journals except JAG.
% \subjclass[2020]{}
\date{29 May 2026}

\dedicatory{}
%    The "communicated by" line appears only in PROC and JAG.
%\commby{}

%    Abstract is required.
\begin{abstract}
  % This paper investigates the greatest lower bound of area of star-shaped Kakeya sets $K$. 
  In 1971, Cunningham proved that every star-shaped Kakeya set $E\subset\mathbb{R}^2$ satisfies $|E| \geq \pi/108$. In this paper, we show that Cunningham's bound is not optimal and can be improved to $|E| \geq \pi/98$.
  % \textbf{Keywords:} star-shaped, Kakeya set, lower bound
\end{abstract}

\maketitle
% \tableofcontents
% A Kakeya set in the plane is a set that allows a unit line segment to rotate continuously by 180 degrees with end reversal.
% \vspace{-2mm}
\section{Introduction}
The Kakeya needle problem~\cite{999733023102121} asks: ``What is the minimal area of a region in which a unit needle can be continuously rotated through 180 degrees with its ends reversed?" 
% The problem was proposed by Sōichi Kakeya in 1917. 
In 1928, Besicovitch~\cite{Besicovitch1928} showed that such sets can have arbitrarily small area. A related question posed by Cunningham \cite{Cunningham_1971} considers sets that contain a unit line segment in every direction (without requiring continuous rotation), with the additional requirement that the set is \textit{star-shaped} (that is, there exists a point $O$ in the set such that for every $x$ in the set, the line segment $Ox$ lies entirely within the set). We will call such a set a \textit{star-shaped Kakeya set}.
% (that is, there exists a point $O$ in the set such that for every $x$ in the set, the line segment $Ox$ lies entirely within the set).
% \begin{definition}
% A set is \textit{star-shaped} if there exists a point $O$ in the set (called a center) such that for every $x$ in the set, the line segment $Ox$ lies entirely within it.
% \end{definition}
% \begin{definition}
% A set is called a \textit{star-shaped Kakeya set} if it contains a unit line segment in every direction and is also star-shaped.
% \end{definition}
Cunningham~\cite{Cunningham_1971} showed that every star-shaped Kakeya set $E$ allowing continuous rotation (hence measurable) has positive area and satisfies $|E|\geq \pi/108$. On the other hand, Cunningham and Schoenberg~\cite{Cunningham_Schoenberg_1965} showed that $\inf|E|\leq \frac{(5-2\sqrt{2})}{24}\pi = (0.09048\cdots)\pi$, by generalizing Kakeya's deltoid construction. 

Cunningham's proof proceeds by decomposing the set $E$ into two parts using a cutoff circle centered at $O$, and establishing his lower bound by combining the estimates for both parts. In the present paper, we improve Cunningham's lower bound to $ |E|\geq \pi/98 $ by establishing lower bounds for the circular cross-sections of $E$. \linebreak
\indent As noted in \cite{Cunningham_1971}, a star-shaped Kakeya set \(E\) need not be measurable. For generality, we consider all such sets and use the notation \(\mathcal{L}_2^*(E)\) to denote the \textit{Lebesgue outer measure} of \(E\). For the circular cross-sections \(E \cap S_r\), we use \(\mathcal{H}_1^*(E\cap S_r)\) to denote their \textit{one-dimensional Hausdorff outer measure}. For direction sets \(A \pmod \pi \subset [0, \pi)\), we use \(\mathcal{L}_1^*(A)\) to denote the \textit{one-dimensional Lebesgue outer measure}. For $\alpha_1, \alpha_2 \in A$, we use $|\alpha_1 - \alpha_2|$ to denote $\alpha_1 - \alpha_2\pmod\pi$. To avoid ambiguity, we will assume without loss of generality that $\alpha_1 > \alpha_2$ and will rotate the coordinate system so that $\alpha_2 = 0$. When the set under consideration is measurable, we simply use \(|\cdot|\) to denote its measure (such as \(|E|\), \(|E \cap S_r|\), and \(|A|\)). These conventions are used throughout the paper. 

In this generality, our main result can be stated as follows.
\begin{theorem}
\label{the:mylowerbound}
Every star-shaped Kakeya set $E$ satisfies
\begin{equation*}
    \mathcal{L}_2^*(E) \geq \frac{\pi}{98}.
\end{equation*}
\end{theorem}

Before proceeding to the proof of Theorem \ref{the:mylowerbound}, we begin by recalling some necessary definitions from \cite{Cunningham_1971}. Without loss of generality, we will assume that $O$ is the origin. Let $S_r$ be the circle of radius $r$ centered at $O$ and let $B_r$ be the \textit{open} disk with the same center and radius. A unit line segment pointing in direction $\alpha$ (i.e., forming an angle $\alpha$ with the $x$-axis) will be called a \textit{needle} and will be denoted by $l_\alpha$. Let $\Delta_\alpha$ denote the \textit{closed} triangle with base $l_\alpha$ and $O$ being a vertex. The height of the triangle $\Delta_\alpha$ will be denoted by $\delta(\Delta)$ (when $\delta(\Delta)=0$, $\Delta_\alpha$ is understood as a line segment). Let $E\subset\mathbb{R}^2$ be a star-shaped Kakeya set. We will always assume that $l_\alpha\subset E$. Thus, $\Delta_\alpha\subset E$. Denote $\Delta_\alpha^{\text{\rm ext}}=\Delta_\alpha\cap B_r^c$ and $\Delta_\alpha^{\text{\rm int}}=\Delta_\alpha\cap B_r$, the parts of $\Delta_\alpha$ outside and inside $B_r$ respectively. We will call $\Delta_\alpha$ and $\Delta_\beta$ ``disjoint" if their interiors are disjoint, that is,  $\mathring\Delta_\alpha\cap\mathring\Delta_\beta=\varnothing$. When the context is clear, we will often omit $\alpha$ and simply write $\Delta,\Delta^{\text{\rm ext}},\Delta^{\text{\rm int}}$ in place of $\Delta_\alpha,\Delta_\alpha^{\text{\rm ext}},\Delta_\alpha^{\text{\rm int}}$ respectively.
% In 1965~\cite{Cunningham_Schoenberg_1965}, Cunningham and Schoenberg gave an upper bound for the constant $K$ of star-shaped Kakeya sets: $K \leq (5-2\sqrt{2})\pi/24 = (0.09048\cdots)\pi$. This paper improves Cunningham's lower bound to $ K \geq \pi/91 $ using cross-sectional integrals.In 1971~\cite{Cunningham_1971}, Cunningham proved the lower bound $K \geq \pi/108$.
% We will restate the definition of a star-shaped Kakeya set in Section 2. We define the Kakeya constant $K$ to be the greatest lower bound of the areas of star-shaped Kakeya sets. The continuous rotation condition is not required for either Cunningham's lower bound or the one established here. 
\section{Generalizing Cunningham's Lower Bound}
\label{sec:2}
% To make this part of your introduction more precise and clear, you can explicitly state the structure and purpose of the section. Here is a revised version:
In this section, we extend Cunningham's method for bounding the area of star-shaped Kakeya sets by establishing a bi-parametrized lower bound that depends on the direction set \( A \subset [0, \pi) \) and the cutoff radius \( r \). Specifically, we prove the following:

% In this section, we present an extension of Cunningham's method for bounding the area of star-shaped Kakeya sets. The extension is by introducing a bi-parametrized lower bound that depends on the measure of the direction set \( A \subset [0, \pi) \) and the cutoff radius \( r \). Specifically, we prove the following:
\begin{theorem}
   Let $r\in[0.15,0.5]$ and let $A\subset[0, \pi).$ Then we have
    \begin{equation}
        \mathcal{L}_2^*(\bigcup_{\alpha\in A}\Delta_\alpha) \geq \frac{\mathcal{L}_1^*(A)}{4}f(r),
        \label{eq:phiareamin}
    \end{equation}
    where $f(r) := \frac{1}{2}r(2r-1)^2$.
    \label{the:phiareamin}
\end{theorem}

Cunningham’s lower bound in \cite{Cunningham_1971}*{Theorem 2} corresponds to the special case \( A = [0, \pi) \) (thus \( \mathcal{L}_1^*(A) = \pi \)) and \( r = 1/6 \), in which Theorem \ref{the:phiareamin} yields the following
% \[
% |E| \geq \frac{\pi}{4} f(1/6) = \frac{\pi}{108}.
% \]
\begin{theorem}[\cite{Cunningham_1971}*{Theorem 2}]
\label{the:Cunningham's lowerbound}
    Every star-shaped Kakeya set $E$ satisfies
    \begin{equation*}
        |E|\geq\frac{\pi}{108}.
    \end{equation*}
\end{theorem}
In comparison with Cunningham's proof of Theorem \ref{the:Cunningham's lowerbound}, Theorem \ref{the:phiareamin} allows one to vary both $A$ and $r$, thereby opening the possibility of optimizing over these parameters to improve the final lower bound.

% When the context is clear, we omit $\alpha$ and simply write $\Delta,\Delta^{\text{\rm ext}},\Delta^{\text{\rm int}}$.

% Now we give a preliminary lower bound for the area of the union of such triangles $\Delta_\alpha$. The proof is by extending that of \cite{Cunningham_1971}*{Theorem 2}, where the case $A = [0, \pi)$ and $r = 1/6$ was considered. 

% Next, we move on to Theorem \ref{the:phiareamin}. Taking $\mathcal{L}_1^*(A)=\pi$ and $r = 1/6$, we can obtain the result of Cunningham \cite{Cunningham_1971}*{Theorem 2} by Theorem \ref{the:phiareamin}.

% \begin{theorem}
%     If $A\subset[0,\pi)$ has Lebesgue measure $\mathcal{L}_1^*(A)$, then for any $r\geq0.15$, we have
%     \begin{equation}
%         \left| \bigcup_{\alpha\in A}\Delta_\alpha \right| \geq \frac{\mathcal{L}_1^*(A)}{4}f(r).
%     \end{equation}
%     where $f(r) := \frac{1}{2}r(2r-1)^2$.
%     \label{the:phiareamin}
% \end{theorem}

To prove Theorem \ref{the:phiareamin}, we start by invoking a technical lemma from \cite{Cunningham_1971}. Without loss of generality, we will assume that $\delta(\Delta)\in [0, \frac{\mathcal{L}_1^*(A)}{2}f(r)]$, since otherwise $\delta(\Delta) >\frac{\mathcal{L}_1^*(A)}{2}f(r)$ would immediately give rise to a $\Delta$ of area $\frac12\delta(\Delta)>\frac{\mathcal{L}_1^*(A)}{4}f(r)$. Under this assumption, $\delta(\Delta)<r$ and so the angle $\arcsin(\delta(\Delta)/r)$ is well defined. 
\begin{figure}[!ht]
    \centering
   \begin{tikzpicture}[scale=7, line join=round, font=\normalsize]
       % 圆心 O (不标记坐标)
        \coordinate (O) at (0,0);
        
        % 计算等腰三角形顶点 (方向 30° 单位针)
        \pgfmathsetmacro{\angle}{90}      % 方向角
        \pgfmathsetmacro{\halfbase}{0.5}  % 单位针一半长度 (1/2)
        \pgfmathsetmacro{\height}{0.1}    % 三角形高度 (自定义值，满足等腰)
        
        % 计算单位针端点 (以 O 为顶点，对称分布)
        \coordinate (H) at (\angle:\height);
        \coordinate (A) at ($(H) + ({\angle+90}:\halfbase)$);
        \coordinate (B) at ($(H) + ({\angle-90}:\halfbase)$);
        % 计算单位针位置 (垂直于角平分线)
        \coordinate (C) at ($(H) + ({\angle+90}:\halfbase+0.2)$);
        \coordinate (D) at ($(H) + ({\angle-90}:\halfbase-0.2)$);
        \coordinate (E) at ($(H) + ({\angle+90}:\halfbase+0.55)$);
        \coordinate (F) at ($(H) + ({\angle-90}:\halfbase-0.55)$);
        % 绘制等腰三角形 (OA = OB)
        % \draw[black, thick] (A) -- (O) -- (B) -- cycle node[midway, above] {$l_\alpha$};
        \filldraw[fill=gray!40, draw=black, thick] (A) -- (O) -- (B) -- cycle node[midway, above] {$l_\alpha$};
            
        % 绘制单位针 (长度为 1)
        % \filldraw[draw = black, dashed ,thick, pattern=grid] (C) --(O) -- (D) -- cycle ;
        % \filldraw[pattern=north east lines, draw = black, dashed ,thick] (C) --(O) -- (D) -- cycle ;
        % \filldraw[pattern=dots, draw = black, densely dotted , thick] (E) --(O) -- (F) -- cycle ;
        \filldraw[pattern=north east lines, draw = black, densely dotted , thick] (E) --(O) -- (F) -- cycle ;
        % 标记圆心 O
        \fill (O) circle (0.1pt) node[below] {$O$};
        % 绘制半径为 1/6 的圆
        \draw[thick] (O) circle (1/6);
        % 标记圆
        \node at (-20:0.25) {$S_{r}$};
    \end{tikzpicture}
    \caption{$|\Delta^{\text{\rm ext}}|$ is minimized when $\Delta$ is isosceles}
    \label{fig:triangle}
\end{figure}
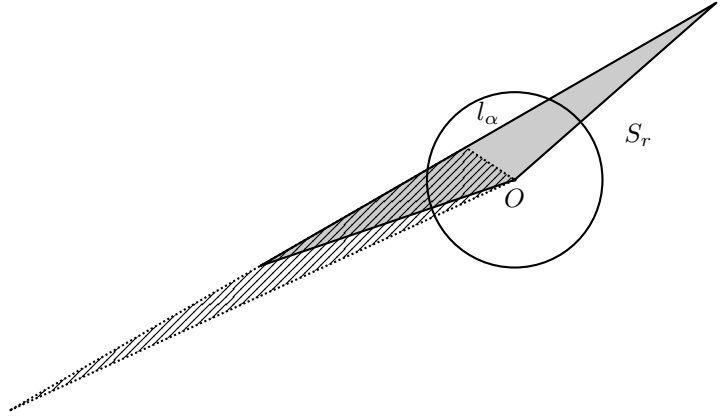

\begin{lemma}
Let $r\in[0.15,0.5]$, and let $\Delta,\Delta^{\text{\rm ext}}$ be as above. Then
\begin{equation}
    |\Delta^{\text{\rm ext}}| \geq f(r)\arcsin(\frac{\delta}{r}),
\end{equation}
% The area of $\Delta^{\text{\rm ext}}$ is at least $f(r)\arcsin(\frac{\delta}{r})$ 
where $\delta=\delta(\Delta)$ and f(r) is as in Theorem \ref{the:phiareamin}.
\label{lem:minofdelta'}
\end{lemma}
\begin{proof}
    For the reader's convenience, we sketch the proof here. Denote $l_\alpha' = l_\alpha\cap B_r^c$. It suffices to show that $|\Delta^{\text{\rm ext}}|$ is minimized when $\Delta$ is isosceles (see Figure \ref{fig:triangle}). This is because by 
    $$
    |\Delta^{\text{\rm ext}}|= \frac{1}{2}l'_\alpha\cdot\delta-\frac{1}{2}|\Delta\cap S_r|r,
    $$ 
    $|l_\alpha'|$ is minimized, and $|\Delta\cap S_r|$ is maximized, when $\Delta$ is isosceles. 
    
    When $\Delta$ is isosceles, by simple trigonometry, we have
    \[
    |\Delta^{\text{\rm ext}}| = \frac{1}{2}\delta - \left( \delta\sqrt{r^2-\delta^2} + \left( \arcsin\frac{\delta}{r} - \arctan(2\delta) \right) r^2 \right).
    \]
To extract a lower bound, consider the function
$$
 h(\delta)=\frac{|\Delta^{\text{\rm ext}}|}{\arcsin{(\frac{\delta}{r})}}.
$$
Direct computation shows that, when $r\geq0.146\cdots$, the minimum of $h(\delta)$ is attained at $\delta=0$.
% This critical value ($r=0.146\cdots$) is the solution for which $h(0)=h(\frac{\pi}{8}f(r))$. For $r \geq 0.146\cdots$, the local minimum is less than the boundary values and thus represents the global minimum.
In particular, when $r\geq0.15$, we have
\begin{align*}
    h(\delta)=\frac{|\Delta^{\text{\rm ext}}|}{\arcsin{(\frac{\delta}{r})}}&\geq
    \lim_{\delta\to0} \frac{|\Delta^{\text{\rm ext}}|}{\arcsin{(\frac{\delta}{r})}}
    = \frac12r(2r-1)^2.
\end{align*}
This yields $|\Delta^{\text{\rm ext}}|\geq f(r)\arcsin(\frac{\delta}{r})$, as desired. 
\end{proof}
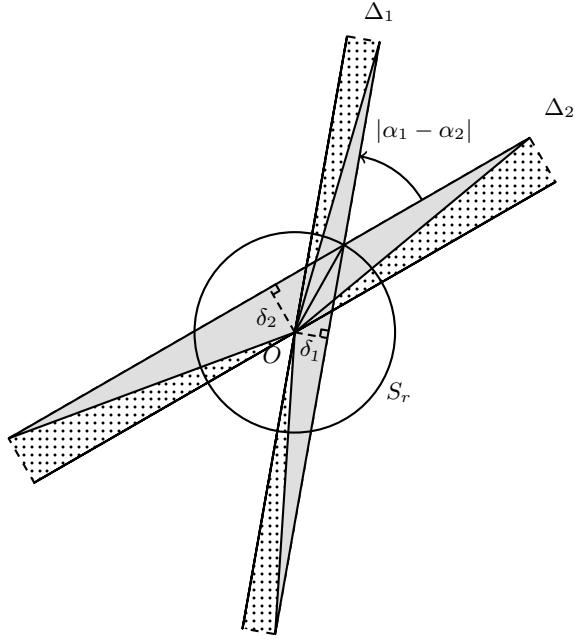
\begin{figure}[!ht]
    \centering
       \begin{tikzpicture}[scale=8.5, line join=round, font=\large]
        % 定义圆心 O (原点)
        \coordinate (O) at (0,0);
        
        % 绘制半径为 1/6 的圆 Γ
        
        \node at (-30:0.20) {$S_r$};
        
        % 定义参数
        \def\angleDiff{-60}      % λ₁ - λ₂ = 30°
        \def\deltaOne{0.055}      % δ₁ = 1.5
        \def\deltaTwo{0.085}      % δ₂ = 1.2
        \def\angle{70}
        % 计算等腰三角形顶点 (方向角 λ₁ 和 λ₂)
        \coordinate (P1) at (\angleDiff:\deltaOne);
        \coordinate (P2) at (\angle:\deltaTwo);
        
        % 计算单位针端点 (长度为1)
        \coordinate (A1) at ($(P1)+({\angleDiff+90}:0.5)$);
        \coordinate (B1) at ($(P1)+({\angleDiff-90}:0.5)$);
        \coordinate (A2) at ($(P2)+(\angle+90:0.5)$);
        \coordinate (B2) at ($(P2)+(\angle-90:0.5)$);
        \coordinate (C1) at ($({\angleDiff+90}:0.5)$);
        \coordinate (D1) at ($({\angleDiff-90}:0.5)$);
        \coordinate (C2) at ($(\angle+90:0.5)$);
        \coordinate (D2) at ($(\angle-90:0.5)$);
        % 绘制等腰三角形
        
        \fill[pattern=dots, draw = black, dashed ,thick] (A1)--(B1)--(D1)--(C1)--cycle;
        \fill[pattern=dots, draw = black, dashed ,thick] (A2)--(B2)--(D2)--(C2)--cycle;
        \fill[gray!40] (A1) -- (O) -- (B1) -- cycle;
        \fill[gray!40] (A2) -- (O) -- (B2) -- cycle;
        \draw[thick] (O) circle (1/6);
        \draw[black, thick] (A1) -- (O) -- (B1) -- cycle;
        \draw[black, thick] (A2) -- (O) -- (B2) -- cycle;
        \draw[black, thick] (C1) -- (O) -- (D1) -- cycle;
        \draw[black, thick] (C2) -- (O) -- (D2) -- cycle;
    
        % 绘制距离线 δ₁ 和 δ₂
        % \draw[densely dashed, thick] (O) -- (P1) node[midway, shift={(-0.15,-0.15)}, rotate=\angleDiff] {$\delta_1$};
        \draw[densely dashed, thick] (O) -- (P1) node[midway, shift={(-0.18,-0.12)}] {$\delta_1$};
        \draw[densely dashed, thick] (O) -- (P2) node[midway, shift={(-0.24,0.1)}] {$\delta_2$};
        \pic [draw, angle radius=1mm, angle eccentricity=1.55, thick,
                  "{}", font=\large] {right angle = B2--P2--O};
        \pic [draw, angle radius=1mm, angle eccentricity=1.55, thick,
                  "{}", font=\large] {right angle = A1--P1--O};
        % 定义两条直线（必须命名路径）
        \path[name path=line1] (B1) -- (A1);         % 第一条直线 OA1
        \path[name path=line2] (B2) -- (A2);      % 第二条直线
    
        \path[name intersections={of=line1 and line2, by={P}}];
        % 标记角度差 λ₁ - λ₂
        \pic [draw, thick,->,angle radius=12mm, angle eccentricity=2,
                "{$|\alpha_1 - \alpha_2|$}"] {angle = B2--P--A1};
        \draw[black, thick] (P)--(O);
        % 标记点 O
        \fill (O) circle (0.pt) node[shift={(-0.45,-0.05)}] {$O$};
        \node at ($(B2)+(0.05,0.0)$) {$\Delta_2$};
        \node at ($(A1)+(0.05,0.0)$) {$\Delta_1$};
    \end{tikzpicture}
    \caption{When $|\alpha_1 - \alpha_2| \geq \arcsin (\delta_1 / r)  + \arcsin (\delta_2 / r)$, $\Delta_{\alpha_1}^{\text{\rm ext}}$ and $\Delta_{\alpha_2}^{\text{\rm ext}}$ are disjoint.} 
    \label{fig:differanlge}
\end{figure}
\begin{lemma}
    \label{lem:sepreate}
    Let $r\in[0,1/2]$ and let $l_{\alpha_1}$, $l_{\alpha_2}$ be two needles. If $|\alpha_1 - \alpha_2| \geq \arcsin (\delta_1 / r)  + \arcsin (\delta_2 / r)$, then $ \mathring\Delta_{\alpha_1}^{\text{\rm ext}} \cap \mathring\Delta_{\alpha_2}^{\text{\rm ext}} = \varnothing$, where $\delta_1=\delta(\Delta_1), \delta_2=\delta(\Delta_2)$.
\end{lemma}
\begin{proof}
Fix $r$, $\alpha_1$ and $\alpha_2$, and write $\Delta_i$ $(i = 1,2)$ in place of $\Delta_{\alpha_i}$. The shaded rectangles (see Figure \ref{fig:differanlge}) represent the union of triangles formed by $\Delta_1,\Delta_2$. When $|\alpha_1 - \alpha_2| > \arcsin(\delta_1 / r) + \arcsin(\delta_2 / r)$, the shaded rectanges are disjoint outside $B_r$. When $|\alpha_1-\alpha_2|=\arcsin(\delta_1 / r) + \arcsin(\delta_2 / r)$, the shaded rectangles intersect at a single point on $S_r$, as illustrated in Figure \ref{fig:differanlge}. 
\end{proof}

\begin{proof}[Proof of Theorem \ref{the:phiareamin}.]
% Next, we consider $\alpha \pmod{\pi}$ as $\alpha$, that is, $\alpha \in [0,\pi)$.
Let $A$ be a subset of $[0,\pi)$ of outer measure $\mathcal{L}_1^*(A)$. Fix $r\geq0.15$ and fix a small $\varepsilon>0$. By the Kakeya property, we can associate each $\alpha$ with a needle $l_\alpha \subset E.$ Denote $\delta(\alpha)=\delta(\Delta_\alpha)$. For fixed $r$ and $\varepsilon$, define intervals:
\[
I_\alpha = \left( \alpha - 2\arcsin\left(\frac{1}{1-\varepsilon}\frac{\delta(\alpha)}{r}\right),\  \alpha + 2\arcsin\left(\frac{1}{1-\varepsilon}\frac{\delta(\alpha)}{r}\right) \right) \pmod{\pi}.
\]
(If $\delta(\Delta_\alpha) =0$, set $I_\alpha =\varnothing$.) We now select a sequence $\{\alpha_n\}$ of $\alpha$. For simplicity, we will denote $I_{\alpha_n}$ by $I_n$, $\Delta_{\alpha_n}$ by $\Delta_n$, and $\delta(\alpha_n)$ by $\delta_n$. Set
\begin{equation}
    \label{eq:E_A}
    E_{A}= \bigcup_{\alpha\in A}\Delta_\alpha\subset E.
\end{equation}
The sequence $\{\alpha_n\}$ (and the associated needles $\{l_{\alpha_n}\}$) is selected as follows:
\begin{enumerate}
    \item If $A\neq\varnothing$, choose $\alpha_1 \in A$ with $\delta_1 > (1-\varepsilon) \sup_{\alpha \in A} \delta(\alpha)$;
    % \item Choose $\alpha_2 \in A \setminus I_1$ with $\delta_2 > (1-\varepsilon) \sup_{\alpha \in A \setminus I_1} \delta(\alpha)$. (If $\alpha_2$ does not exist, then $A\subset I_1$);
    \item For $k\geq2$, if $A\setminus\cup_{n=1}^{k-1}I_n \neq \varnothing$, choose $\alpha_k \in A\setminus\cup_{n=1}^{k-1} I_n$ with
    $$
    \delta_k > (1-\varepsilon) \sup_{\alpha \in A\setminus\cup_{n=1}^{k-1} I_n} \delta(\alpha);
    $$
    \item Continue this selection process unless $A\setminus \cup_{n=1}^{k-1} I_n =\varnothing$.
\end{enumerate}

Notice that for any $\alpha\notin I_1$, since $\delta_1>(1-\varepsilon)\delta(\alpha)$, we have
\begin{equation*}
    2\arcsin(\frac{1}{1-\varepsilon}\frac{\delta_1}{r})>\arcsin(\frac{\delta_1}{r})+\arcsin(\frac{\delta(\alpha)}{r}).
\end{equation*}
Therefore, by Lemma \ref{lem:sepreate}, $\Delta_\alpha^{\text{\rm ext}}$ and $\Delta_1^{\text{\rm ext}}$ are disjoint. Repeating the argument, it is easy to see that $\{\Delta^{\text{\rm ext}}_n\}_{n\geq1}$ are pairwise disjoint.

Two cases arise in the selection process:

\textbf{Case 1: } The selection process terminates in finite steps. In this case, there exists a finite integer $k\geq1$ such that $A \subset \bigcup_{n=1}^k I_n$. Thus, since $\bigcup_n\Delta_n^{\text{\rm ext}}\subset E_A$, we have
\begin{alignat}{2}
    \mathcal{L}_2^*(E_A) &\geq \sum_{n=1}^k|\Delta_n^{\text{\rm ext}}|
    && \text{(by Lemma \ref{lem:sepreate})} \label{eq:geqchain} \\
    &\geq \sum_{n=1}^k f(r)\arcsin\left(\frac{\delta_n}{r}\right)
    && \text{(by Lemma \ref{lem:minofdelta'})} \notag \\
    &\geq \frac{2}{\pi}\frac{\arcsin(1-\varepsilon)}{4(1-\varepsilon)} f(r) \sum_{n=1}^k 4(1-\varepsilon) && \arcsin\left(\frac{1}{1-\varepsilon}\frac{\delta_n}{r}\right) \notag \\
    &= \frac{2}{\pi}\frac{\arcsin(1-\varepsilon)}{4(1-\varepsilon)} f(r) \sum_{n=1}^k|I_n|
    && \text{(by the definition of }I_n) \notag \\
    &\geq \frac{2}{\pi}\frac{\arcsin(1-\varepsilon)}{4(1-\varepsilon)} f(r) \mathcal{L}_1^*(A)
    && (\text{since } \mathcal{L}_1^*(A) \leq \left|\bigcup_{n=1}^k I_n\right|), \notag
\end{alignat}
where in the third inequality, we have used $\frac{\arcsin(mx)}{\arcsin(x)}\geq \frac{2}{\pi}\arcsin(m)$, when $x, m \in (0,1]$. Taking $\varepsilon \to 0$, we obtain:
\begin{equation}
     \mathcal{L}_2^*(E_A) \geq \frac{\mathcal{L}_1^*(A)}{4}f(r).
\end{equation}
This shows that (\ref{eq:phiareamin}) holds in Case 1.

\textbf{Case 2:} The selection process does not terminate. In this case, since $E_A\cap B_r^c \supset \bigcup_n\Delta_n^{\text{\rm ext}}$, denoting $C_\varepsilon = \frac{2}{\pi}\frac{\arcsin(1-\varepsilon)}{1-\varepsilon}$ and noting that $C_\varepsilon\to1$ as $\varepsilon\to0$, we have 
\begin{align}
    \mathcal{L}_2^*(E_A\cap B_r^c) &\geq\sum_{n=1}^\infty |\Delta_n^{\text{\rm ext}}| 
    && \label{eq:EacapBr} \\ 
    &\geq \frac{C_\varepsilon}{4}f(r)\sum_{n=1}^\infty4(1-\varepsilon)\arcsin\left(\frac{1}{1-\varepsilon}\frac{\delta_n}{r}\right)
    && (\text{similar to \eqref{eq:geqchain}}). \label{eq:b}
\end{align}
Denote the last series by  
$$
b(r) = \sum_{n=1}^\infty4(1-\varepsilon)\arcsin\left(\frac{1}{1-\varepsilon}\frac{\delta_n}{r}\right).
$$
For simplicity, we will simply write $b$ for $b(r)$ below. Then we have
\begin{align}
    |\bigcup_n I_n|\leq b,
    \label{eq:cupIn}
\end{align}
%     |\bigcup_n I_n|\leq\frac{4}{C_\varepsilon f(r)}b && (\text{where } b=b(r)),
%     \label{eq:cupIn}
% \end{align} 
since $|\cup_n I_n|\leq\sum_n |I_n|$. Therefore, by (\ref{eq:cupIn}), 
\begin{equation}
    \mathcal{L}_1^*(A \setminus \bigcup_{n=1}^\infty I_n) \geq \mathcal{L}_1^*(A) -  b.\label{eq:anlgeofdeltavanished}
\end{equation}
Now consider two subcases: 
    
\textbf{Subcase 2a:} $b = \infty$.
In this case, the desired bound (\ref{eq:phiareamin}) follows immediately, since (\ref{eq:EacapBr}) and (\ref{eq:b}) together imply $\mathcal{L}_2^*(E_A)=\infty$.

\textbf{Subcase 2b:} $b < \infty$.
In this case, the convergence of the series in (\ref{eq:b}) implies $\delta_k \to 0$. Since $(1-\varepsilon)\sup_{\alpha \in A \setminus (\cup_n I_n)} \delta(\alpha) \leq \delta_k$ by the selection process, we have $\sup_{\alpha \in A \setminus (\cup_n I_n)}\delta(\alpha) = 0$. Thus the set $E_A$ contains a needle passing through $O$ in every direction $\alpha\in A \setminus \bigcup_{n=1}^\infty I_n=:A_{0}$. By (\ref{eq:anlgeofdeltavanished}) and Lemma \ref{lem:polmeasure}, the union of these needles satisfies 
\begin{equation}
    \label{eq:linemeasure}
    \mathcal{L}_2^*(B_r\cap\bigcup_{\alpha\in A_0}l_\alpha)\geq\frac{1}{2} r^2 \left( \mathcal{L}_1^*(A) -  b \right).
\end{equation}
Since $\{\Delta_n^{\text{\rm ext}}\}$ are disjoint, we have
\begin{align*}
\mathcal{L}_2^*(E_A) &= \mathcal{L}_2^*( E_A\cap B_r) + \mathcal{L}_2^*( E_A\cap B_r^c)\notag\\
&\geq \frac{1}{2} r^2 \left( \mathcal{L}_1^*(A) -  b \right) + \frac{C_\varepsilon}{4}f(r)b \\
&= \frac{\mathcal{L}_1^*(A)}{2} r^2 + \left(\frac{C_\varepsilon}{4}f(r)-\frac{1}{2}r^2\right) b.
\end{align*}
Since 
$$
\frac{C_\varepsilon}{4}f(r)-\frac{1}{2}r^2 \leq \frac{1}{4}f(r)-\frac{1}{2}r^2
$$ 
and $f(r) - {2r^2} \leq 0$ when $r\geq1-\frac{\sqrt{3}}{2}$, and since $\mathcal{L}_2^*(E_A)\geq \frac{C_\varepsilon}{4}f(r)b$ by \eqref{eq:b}, we obtain:
\begin{align*}
    \mathcal{L}_2^*(E_A) \geq \frac{\mathcal{L}_1^*(A)}{2} r^2 + \left(1 -\frac{2r^2}{C_\varepsilon f(r)}\right) \mathcal{L}_2^*(E_A).
\end{align*}
Taking $\varepsilon\to0$, this implies $\mathcal{L}_2^*(E_A) \geq \frac{\mathcal{L}_1^*(A)}{4}f(r)$.
This completes the proof of Theorem \ref{the:phiareamin}.
\end{proof}
% With Lemma \ref{lem:minofdelta'} and Theorem \ref{the:phiareamin}, we recover the result of Cunningham \cite{Cunningham_1971}*{Theorem 2} and \cite{Cunningham_1971}*{Lemma}. Taking $\mathcal{L}_1^*(A)=\pi$ and $r = 1/6$ we obtain:
% \begin{corollary}\cite{Cunningham_1971}*{Lemma}
% For $r = 1/6$, the area of $\Delta^{\text{\rm ext}}$ is at least $\frac{1}{27}\arcsin(6\delta)$.
% \label{lem:minareaofDelta'}
% \end{corollary}
The function $f(r)$ in Theorem \ref{the:phiareamin} is maximized at $r=1/6$, which gives Cunningham's lower bound when $A=[0, \pi)$. In order to improve Cunningham's lower bound, the key departure is to improve the lower bound for $|E_A\cap B_r|$. More precisely, we have: 
\begin{lemma}
    \label{lem:inplusout}   
    Let $A\subset[0,\pi)$ and let $r\geq0.15$. If
    \begin{equation}
        \mathcal{L}_2^*(E_A\cap B_r) \geq s_0,
    \end{equation}
    then
    \begin{equation}
    \label{eq:inplusout}
        \mathcal{L}_2^*(E_A) \geq \frac{\mathcal{L}_1^*(A)}{4}f(r) + \left(1 - \frac{f(r)}{2r^2}\right) s_0,
    \end{equation}
    where $E_A$ is as in (\ref{eq:E_A}).
\end{lemma}
\begin{proof}
    We use the same notation as in the proof of Theorem \ref{the:phiareamin}. If $A$ is covered by finitely many intervals $I_n$, then $\mathcal{L}_2^*( E_A )= \mathcal{L}_2^*( E_A\cap B_r^c) + \mathcal{L}_2^*( E_A\cap B_r )\geq\frac{\mathcal{L}_1^*(A)}{4}f(r) + s_0$. Thus (\ref{eq:inplusout}) follows. Otherwise, two cases arise:

    \textbf{Case 1: }$s_0 \geq \frac{1}{2}r^2 \left( \mathcal{L}_1^*(A) - \frac{4}{f(r)} b \right)$.

    In this case, we have $\mathcal{L}_2^*( E_A\cap B_r^c)\geq b \geq \frac{\mathcal{L}_1^*(A)}{4}f(r) - \frac{f(r)}{2r^2}s_0$. Therefore,  
    \begin{align*}
        \mathcal{L}_2^*(E_A)&\geq b + s_0\\ 
        &\geq \frac{\mathcal{L}_1^*(A)}{4}f(r) + \left(1 - \frac{f(r)}{2r^2}\right) s_0.
    \end{align*}

    \textbf{Case 2: }$s_0 \leq \frac{1}{2}r^2 \left( \mathcal{L}_1^*(A) - \frac{4}{f(r)} b \right)$.

    In this case, we have $b \leq \frac{\mathcal{L}_1^*(A)}{4}f(r) - \frac{f(r)}{2r^2}s_0$. Therefore, 
    \begin{align*}
        \mathcal{L}_2^*(E_A)&\geq b + \frac{1}{2}r^2 \left( \mathcal{L}_1^*(A) - \frac{4}{f(r)} b \right) \\&\geq \frac{\mathcal{L}_1^*(A)}{4}f(r) + \left(1 - \frac{f(r)}{2r^2}\right) s_0.
    \end{align*}
    Combining these two cases, we complete the proof.
\end{proof}

\section{Improving the Lower Bound via cross-sections}
\label{sec:3}
Cunningham partitioned the star-shaped Kakeya set $E$ using a circle of radius $1/6$. In this section, we show that his estimate for the inner part $|E\cap B_r|$ can be improved using circular cross-sections. 

Using polar coordinates, the Lebesgue outer measure of $E$ satisfies the inequality:
\begin{equation}
    \mathcal{L}_2^*(E) \geq \int_{\mathbb{R}^+}^* \mathcal{H}_1^*(E \cap S_r)\,  dr,
    \label{eq:fubini}
\end{equation}
where $\mathbb{R}^+$ denotes the set $\{x\in\mathbb{R}\mid x\geq0\}$ (See Lemma \ref{lem:polar_fubini}.). Here the upper integral $\int^*$ is defined as:
\[
\int^* f \, dx= \inf \left\{ \int g \, dx: g \text{ is measurable and } g \geq f \right\}.
\]
% (A detailed proof of (\ref{eq:fubini}) is provided in Lemma \ref{lem:polar_fubini}.)
By establishing a lower bound for $\mathcal{H}_1^*(E \cap S_r)$, we will use (\ref{eq:fubini}) to show that:
\begin{equation*}
\mathcal{L}_2^*(E) \geq \frac{\pi}{98}.
\end{equation*}
thus proving Theorem \ref{the:mylowerbound}.

% We may assume $\delta < \pi/50$ for all needle positions; otherwise $|E| \geq \delta_0/2 \geq \pi/100$. Let $a = \pi/50$ and $r_0 = 2/7 > a$. We establish a uniform lower bound for the inner area. For a needle in direction $\alpha$, let $\Gamma_{\alpha,1}$ and $\Gamma_{\alpha,2}$ be the arcs (ordered by length) formed by $\Delta_\alpha \cap S_r$ (Figure \ref{fig:arc}). 
Fix $r\in [0,1/2]$. As in the proof of Theorem \ref{the:phiareamin}, we may assume that every needle satisfies 
$$
\delta < \pi/49 =:a.
$$ 
Set $r_0 = 1/4 > a$. For a needle in direction $\alpha$, let $\Gamma_{\alpha,1}$ and $\Gamma_{\alpha,2}$ be the connected components of $\Delta_\alpha \cap S_r$ (each of which will be called an ``arc"), ordered in such a way that $|\Gamma_{\alpha,1}| \geq |\Gamma_{\alpha,2}|$ (with the convention that $\Gamma_{\alpha,2} = \varnothing$ if only one arc is formed; see Figure \ref{fig:arc}).

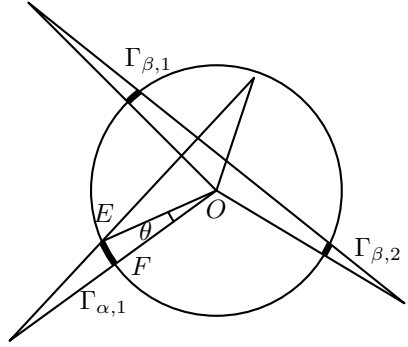
\begin{figure}[!ht]
    \centering
    \begin{tikzpicture}[scale=5.5, line cap=round, line join=round]
        % 定义圆的参数
        \def\radius{1/3}
        \coordinate (O) at (0,0); % 圆心
    
        % 绘制圆（细线）
        \draw[thick] (O) circle (\radius);
        % 定义第一个三角形（上方）
        \coordinate (A) at (0,0);
        \coordinate (B) at (-0.5,+0.5);
        \coordinate (C) at (0.5,-0.3);
        \draw[thick] (A) -- (B) -- (C) -- cycle;
    
        % 定义第二个三角形（下方）
        \coordinate (D) at (0,0);
        \coordinate (E) at (0.10,0.30);
        \coordinate (F) at (-0.55,-0.4);
        \coordinate (P1) at (-0.304499, -0.135615);
        \coordinate (P2) at (-0.269579, -0.196057);
        
        \draw[thick] (D) -- (E) -- (F) -- cycle;
        \draw[thick] (O)--(P1);
        % \pic [draw, thick, angle radius=7mm, angle eccentricity=1.55, 
        %           "{$\theta$}", font=\large] {angle = P1--O--P2};
        \pic [draw, thick, angle radius=10mm, angle eccentricity=1.45,
      "{$\theta$}", font=\Large] {angle = P1--O--P2};
        % 绘制圆与第一个三角形相交的弧（粗黑线）
        \begin{scope}
            \clip (A) -- (B) -- (C) -- cycle;
            \draw[line width=2.5pt] (O) circle (\radius);
            % \draw[line width=2.5pt,gray] (O) circle (\radius);
        \end{scope}

        % 绘制圆与第二个三角形相交的弧（粗黑线）
        \begin{scope}
            \clip (D) -- (E) -- (F) -- cycle;
            \draw[line width=2.5pt] (O) circle (\radius);
        \end{scope}

        \fill (O) circle (0.1pt) node[below] {$O$};
        \draw (-0.25,0.385) node [anchor=north west][inner sep=0.75pt]  [font=\normalsize,xscale=1,yscale=1]  {$\Gamma_{\beta,1}$};
        \draw (0.5,-0.2) node [anchor=south east][inner sep=0.75pt]  [font=\normalsize,xscale=1,yscale=1]  {$\Gamma_{\beta,2}$};
        \draw (-0.4,-0.3) node [anchor=north west][inner sep=0.75pt]  [font=\normalsize,xscale=1,yscale=1]  {$\Gamma_{\alpha,1}$};
        \node at (-0.35,-0.12) {$E$};
        \node at (-0.27,-0.25) {$F$};
    \end{tikzpicture}
    \caption{Arcs formed by triangle-circle intersection}
    \label{fig:arc}
\end{figure}

In order to establish a uniform lower bound for the inner part, we select arcs using a Vitali-type covering argument, as follows:
\begin{enumerate}
     \item[(1)] Set $\mathscr{A}_1 = \{\Gamma_{\alpha,i}\}$ $(i = 1,2)$. If $\mathscr{A}_1\ne\varnothing$, 
      choose $\widetilde{\Gamma}_1 \in \mathscr{A}_1$ satisfying \\ 
      $|\widetilde{\Gamma}_1| \geq (1 - \varepsilon) \sup_{\Gamma \in \mathscr{A}_1} |\Gamma|$;
    \item[(2)] For $k\geq2$, if $\mathscr{A}_k:= \{ \Gamma_{\alpha,i} \mid 
      \Gamma_{\alpha,i} \cap \bigcup_{j=1}^{k-1} \widetilde{\Gamma}_j = \varnothing \} \neq\varnothing$, 
      choose $\widetilde{\Gamma}_k \in \mathscr{A}_k$ satisfying 
      $|\widetilde{\Gamma}_k| \geq (1 - \varepsilon) \sup_{\Gamma \in \mathscr{A}_k} |\Gamma|$;

      \medskip
      \noindent
      if $\mathscr{A}_k=\varnothing$, then we have $$\bigcup_{\alpha,i}\Gamma_{\alpha,i}
      \subset \bigcup _{n=1}^{k-1}\bigl( \frac{2}{1-\varepsilon} + 1 \bigr) \widetilde{\Gamma}_n,
      $$
      since $\Gamma_{\alpha,i} \cap\bigcup _{n=1}^{k-1} \widetilde{\Gamma}_n \neq \varnothing$ and
      $|\widetilde{\Gamma}_n|\geq(1-\varepsilon)\sup_{\Gamma \in \mathscr{A}_n}|\Gamma|$,
      $n=1,\cdots, k-1$.
    % \item[(3)] Set $\mathscr{A}_3 = \{ \Gamma_{\alpha,i} \mid \Gamma_{\alpha,i} \cap \bigcup_{j=1}^{k} \widetilde{I}_j = \varnothing \}$. Choose $\widetilde{\Gamma}_k \in \mathscr{A}_k$ similarly. If $\alpha_3$ does not exist, $\mathscr{A}_k=\varnothing$ and $\bigcup_{\alpha,i}\Gamma_{\alpha,i}\subset \bigcup _{n=1}^k\left( \frac{2}{1-\varepsilon} + 1 \right) \widetilde{\Gamma}_n$;
    \item[(3)] Continue this selection process unless $\mathscr{A}_k =\varnothing$.
    % \begin{center}
    %     $\cdots$
    % \end{center}                            
    % \item[(k)] Set $A_k = \{ \Gamma_{\alpha,i} \mid \Gamma_{\alpha,i} \cap \bigcup_{j=1}^{k-1} \widetilde{I}_j = \varnothing \}$. Choose $\widetilde{\Gamma}_k \in A_k$ similarly.
    % \begin{center}
    %     $\cdots$
    % \end{center}   
\end{enumerate}
Note that the union of the arcs constituting $\Delta_\alpha \cap S_r$ satisfies:
\begin{equation}
    \label{eq:vitalitype}
    \bigcup_{\alpha,i} \Gamma_{\alpha,i} \subset \bigcup_k \left( \frac{2}{1-\varepsilon} + 1 \right) \widetilde{\Gamma}_k,
\end{equation}
where $\left( \frac{2}{1-\varepsilon} + 1 \right) \widetilde{\Gamma}_k$ is the arc with the same midpoint as $\widetilde{\Gamma}_k$ and is $\left( \frac{2}{1-\varepsilon} + 1 \right)$ times as long as $\widetilde{\Gamma}_k$.

Let $\Gamma =\wideparen{EF}$ be an arc in $S_r$. In what follows, we will call the angle $\theta = \angle EOF$ the \textit{central angle} of $\Gamma$ (see Figure \ref{fig:arc}). Slightly abusing notation, we will often identify $\theta$ with the corresponding set of angles $\theta\subset[0,2\pi)$ (In fact, we will show below that $|\theta|\leq\frac{\pi}{2}$); in particular, we have $|\Gamma| = r|\theta|$. With this notation, (\ref{eq:vitalitype}) implies:
\begin{equation}
    \label{eq:thetavitali}
    \bigcup_{\alpha,i}\theta_{\alpha,i}\subset\bigcup_k\left(\frac{2}{1-\varepsilon}+1\right)\widetilde{\theta}_k,
\end{equation}
where $\theta_{\alpha,i}$ and $\widetilde{\theta}_k$ are the central angles of $\Gamma_{\alpha,i}$ and $\widetilde{\Gamma}_k$, respectively. Similarly, $\left( \frac{2}{1-\varepsilon} + 1 \right)\widetilde{\theta}_k$ is the interval with the same center as $\theta_{\alpha,1}$ and is $\left( \frac{2}{1-\varepsilon} + 1 \right)$ times as long as $\widetilde{\theta}_k$.

Define $J_\Gamma = \{\alpha \mid \Gamma_{\alpha,1} \subset \Gamma\}$. Clearly, $\alpha \in J_{\Gamma_{\alpha,1}}$ and $[0,\pi) = \bigcup_{\alpha} J_{\Gamma_{\alpha,1}}$. 
% To put $J_\Gamma$ into $[0,2\pi)$ we construct a map:
% \begin{align*}
% f: J_{\Gamma_{\alpha,1}} &\to [0, 2\pi) \\
% \alpha &\mapsto \alpha && (\text{some case})\\
% \alpha &\mapsto \alpha + \pi && (\text{other case})
% \end{align*}
If $J_{\Gamma_{\alpha,1}}\subset C\theta_{\alpha,1}$ for a suitable constant $C$, then:
\[
\pi = \left| \bigcup_{\alpha} J_{\Gamma_{\alpha,1}} \right| \leq \mathcal{L}_1^*\left(\bigcup_{\alpha,i} C\theta_{\alpha,i}\right)\leq C\left( \frac{2}{1-\varepsilon}+1 \right) \sum_k |\widetilde{\theta}_k|,
\]
where the last series satisfies 
$$
r\sum_k|\widetilde{\theta}_k|=\sum_k |\widetilde{\Gamma}_k| \leq |\bigcup_\alpha \Delta_\alpha \cap S_r|.
$$ 
However, without restrictions on needle position, this bound may fail. Our next objective is to find the critical condition under which this bound holds.

First, we consider the size of $ \theta $ for a fixed radius $ r $. As the needle moves towards infinity, $ |\theta| $ tends to 0. The value of $\theta$ attains its maximum when one endpoint of the needle lies on $S_r$ (see Figure \ref{fig:thetarange}). Thus the size of $ \theta $ satisfies:
\begin{align*}
    0\leq |\theta| \leq \arcsin\left(\frac{a}{r}\right) - \arctan\left(\frac{a}{\sqrt{r^2-a^2}+1}\right)\leq\frac{\pi}{2}.
\end{align*}
\begin{figure}[!ht]
    \centering
       \begin{tikzpicture}[scale=5, line join=round, font=\small]
        % 定义圆心 O (原点)
        \coordinate (O) at (0,0);
        % 绘制半径为 1/6 的圆 Γ
        \coordinate (A) at (-10:1/3);
        \coordinate (B) at ($(A)+({20}:1)$);
        \coordinate (P) at (-60:1/6);
        \coordinate (B1) at ($(A)+({20}:0.6)$);
        \coordinate (A1) at ($(A)+({20}:-0.4)$);
        \coordinate (A2) at ($(P)+({20}:-0.5)$);
        \coordinate (B2) at ($(P)+({20}:0.5)$);
        \draw[|<->|, >=stealth, thick] ($(O)+(-100:0.04)$) -- node[fill=white] {$r$} ($(A)+(-100:0.04)$);
        \begin{scope}
            \clip (A) -- (O) -- (B) -- cycle;
            \draw[line width=2.5pt,gray] (O) circle (1/3);
        \end{scope}
        \node[font=\large] at (5:1/3+0.05) {$\Gamma$};
        % 绘制等腰三角形
        \draw[thick] (O) circle (1/3);
        \draw[black, thick] (A) -- (O) -- (B) -- cycle;
        \draw[|<->|, >=stealth, thick] ($(A)+(-70:0.06)$) -- node[fill=white] {$1$} ($(B)+(-70:0.06)$);
        \pic [draw, thick, ->,>=stealth, angle radius=8mm, angle eccentricity=1.55, 
                  "{$\theta$}", font=\large] {angle = A--O--B};
        \fill (O) circle (0.1pt) node[shift={(-0.2,0.2)}] {$O$};
        \node at (-30:0.40) {$S_{r}$};
    \end{tikzpicture}
    \caption{$\theta$ is maximized for fixed $r$, $\delta$ and $\alpha$}
    \label{fig:thetarange}
\end{figure}
% Consider a fixed arc $ I $ with endpoints $ E $ and $ F $, and central angle $ \angle EOF = \theta $. We now examine the set $ \alpha \in J_\Gamma $ under two scenarios.

Consider a fixed arc \(\Gamma = \wideparen{EF} \). We examine the set \(J_\Gamma = \{\alpha \mid \Gamma_{\alpha,1} \subset \Gamma\}\) under two scenarios. We will assume that $ |\theta| \leq \frac{\pi}{2} - \arctan(2r) $ below, since for $ |\theta| \geq \frac{\pi}{2} - \arctan(2r) $, we have 
$$ 
\frac{|J_\Gamma|}{|\theta|} \leq \frac{\pi}{|\theta|} \leq \frac{\pi}{ \frac{\pi}{2} - \arctan(2r)}.
$$
In this case, we have $J_\Gamma\subset \frac{\pi}{ \frac{\pi}{2} - \arctan(2r)}\theta$.

\textbf{Case 1:} $ l_\alpha \cap B_r \neq \varnothing $. In this case, the critical directions $ \alpha $ for which $ \Gamma_{\alpha,1} = \Gamma $ hold are attained when $ \Delta_\alpha $ is isosceles (see Figure \ref{fig:l-alpha-in-Br}). There are exactly two such directions. Denote the larger one by $ \alpha_1 $, and the smaller one by $ \alpha_2 $. This is because the lines $ l_\alpha $ satisfying both $ \Delta_\alpha \cap B_r = \Gamma $ and $ l_\alpha \cap B_r \neq \varnothing $ pass through the point $ E $ and have one endpoint $ B' $ lying on the ray $ OF $. The longer the segment $ EB' $, the smaller the angle $ \alpha $ between $ EB' $ and $ EA $. Conversely, a shorter $ |EB'| $ results in a larger $ \alpha $. If $ |EB'| $ is smaller than the length shown in Figure \ref{fig:l-alpha-in-Br}, then $ \Gamma = \Gamma_{\alpha,2} $, meaning $ \Gamma_{\alpha,1} \nsubseteq \Gamma $, and thus $ \alpha \notin J_{\Gamma} $. Let $ \delta_0 $ be the distance from $ O $ to $ l_{\alpha_1} $. Then $ |\theta| = \arcsin(\frac{\delta_0}{r}) - \arctan(2\delta_0) $, and $ \alpha_1 = \arcsin(\frac{\delta_0}{r}) $. Therefore, for $ |\theta| \leq \frac{\pi}{2} - \arctan(2r) $, there is a one-to-one correspondence between $ \delta_0(r) $ and $ |\theta| $, implying that an isosceles triangle $ \Delta_\alpha $ can always be found.

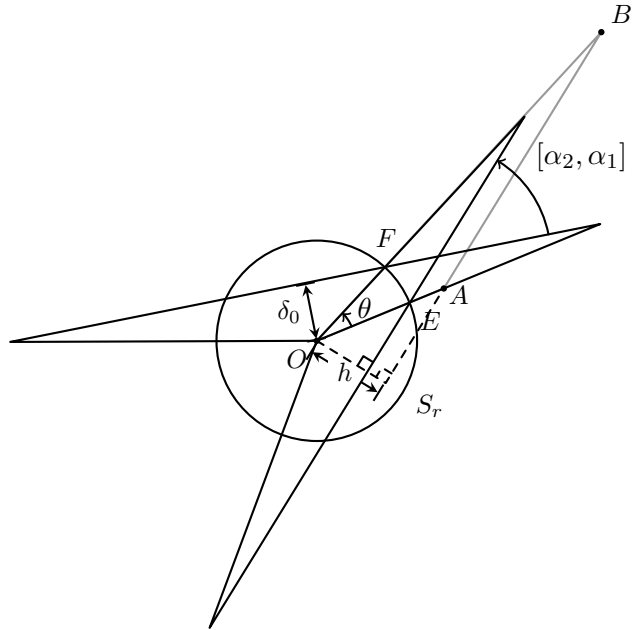
\begin{figure}[!ht]
    \centering
    \scalebox{1}{
        \begin{tikzpicture}[scale=4.5, line cap=round, line join=round]
            % 原点 O
            \coordinate (O) at (0,0);
            
            % 绘制小圆（半径1/3）
            \def\r{1/3}
            \draw[thick] (O) circle (\r);
        
            \coordinate (F) at (1.01021,-0.13472);
            \coordinate (E) at (-0.88490,0.50622);
            \coordinate (I) at (-0.78589,-0.64890);
            \coordinate (H) at (0.97095,0.30079);
            \coordinate (A) at (1.33234,0.41613);
            \coordinate (G) at (0.45258,-0.06037);
            \coordinate (J) at (0.48870,0.04171);
            \coordinate (K) at (0.12797,-0.23622);
            \coordinate (L) at (0.09343,-0.17244);
            \coordinate (P) at (0.06285,0.18575);
            \coordinate (E') at (0.33015,-0.04555);
            \coordinate (F') at (0.31814,0.09941);
            
            % 绘制三角形
            \draw[|<->|, >=stealth, thick] ($(O)+(101.309645991-90:0)$) -- node[left] {$\delta_0$} ($(P)+(101.309645991-90:0)$);   
             \draw[|<->|, >=stealth, thick] ($(O)+(-31.5541271757-90:0.04)$) -- node[fill=white] {$h$} ($(K)+(-31.5541271757-90:0.04)$);      
              
            \draw[gray!60,black!40,thick] (A) -- (G) -- (O) -- cycle;
            \draw[thick] (E) -- (F) -- (O) -- cycle;
            \draw[thick] (H) -- (I) -- (O) -- cycle;
            \draw[thick,dashed] (O)--(L)--(K)--(G);
            % 标记点
            \fill (O) circle (0.3pt) node[shift={(-0.4,0)}] {$O$};
            \fill (A) circle (0.3pt) node[above right] {$B$};
            \fill (G) circle (0.3pt);
            \fill (H) circle (0.3pt) node[above left] {$B'$};
            \fill (I) circle (0.3pt) node[below left] {$A'$};
            \fill (E') circle (0.3pt);
            \fill (F') circle (0.3pt);
            % 角度标记（确保使用正确的库）
            \pic [draw, thick, ->, angle radius=16mm, angle eccentricity=1.55, 
                  "{$|\alpha_1-\alpha_2|$}", font=\large] {angle = F--J--H};
            \pic [draw, thick, ->, angle radius=5mm, angle eccentricity=1.55, 
                  "{$\theta$}", font=\large] {angle = F--O--H};
            \pic [draw, thick, angle radius=1.5mm, angle eccentricity=1.55, 
                  "{}", font=\large] {right angle = H--L--O};
            \pic [draw, thick, dashed, angle radius=1.5mm, angle eccentricity=1.55, 
                  "{}", font=\large] {right angle = A--K--O};
           
            \node at ($(G)+(0.08,-0.05)$) {$A$};
            \node at (26:1/3+0.08) {$F$};
            \node at (-20:1/3+0.05) {$E$};
            \node at (-60:1/3+0.1) {$S_r$};
            \node at (-118:1/3+0.2) [font=\large] {$l_{\alpha_1}$};
            \node at (128:1/3+0.2) [font=\large] {$l_{\alpha_2}$};
        \end{tikzpicture}
    }
    \caption{The critical cases where $\alpha$ does not exceed $[\alpha_1,\alpha_2]$}
    \label{fig:l-alpha-in-Br}
\end{figure}

\textbf{Case 2:} $ l_\alpha \cap B_r = \varnothing $. In this case, the maximum angle $ \alpha $ for which $ \Gamma_{\alpha,1} = \Gamma $ occurs when $\delta(\Delta_\alpha)=a$. The endpoints $ A, B $ of $ l_\alpha $ must lie on the rays $ OE $ and $ OF $ respectively, with $ A, B \in B_r^c $. Without loss of generality, we may assume $OA$ is the $x-$axis. This case only exists if $ OA\geq r $. Consider a direction $ \beta $ such that $ l_\beta \cap B_r = \varnothing $ and $ \Delta_\beta \cap B_r = \Gamma $, as shown in Figure \ref{fig:l-alpha-in-Br}. A necessary and sufficient condition for $ \beta = \alpha_1 $ is that the distance $ h $ from $ O $ to $ l_\beta $ satisfies 
% $ h \leq a $. Simple geometric calculations yield:
\begin{equation}
    h = \frac{\delta_0}{\frac{1}{2} - \sqrt{r^2 - \delta_0^2}} \leq a. \label{eq:h}
\end{equation}
Thus, $ \delta_0 $ must satisfy:
\begin{equation}
    \delta_0 \leq \frac{a(1 - \sqrt{4r^2 + 4a^2 r^2 - a^2})}{2(a^2 + 1)}=:\delta_1(r). \label{eq:deltarange}
\end{equation}
There are exactly two such directions $\beta$ such that $ \delta_\beta = a $. Similarly as above, denote the larger one by $ \beta_1 $, and the smaller one by $ \beta_2 $. Simple geometry (see Figure \ref{fig:l-out-Br}) gives
\begin{equation}
    |OA| = \sqrt{\frac{\left(\frac{2a}{\tan|\theta|} + 1\right) - \sqrt{1 - 4a^{2} + \frac{4a}{\tan|\theta|}}}{2}},
\end{equation}
and $ \beta_1 = \arcsin(\frac{a}{|OA|}) $. When $ \delta_0 $ satisfies \eqref{eq:deltarange}, we have $ \beta_1 \geq \alpha_1 $, $\beta_2 \leq \alpha_2$. Therefore $\beta_2\leq\alpha_2<\alpha_1\leq\beta_1$. Thus $J_\Gamma \subset [\beta_2,\beta_1]$ and we have 
$$ 
\frac{|J_\Gamma|}{|\theta|} \leq \frac{|\beta_1 - \beta_2|}{|\theta|} = \frac{2\arcsin(\frac{a}{|OA|}) -|\theta|}{|\theta|}. 
$$ 
Without a bound on the needle's location, the interior area might tend to 0, yielding $ \mathcal{H}_1^*(E \cap S_r) \geq 0 $. This is because as $ |\theta| \to 0 $, $ |OA| \to \infty $, then 
$$ 
\frac{2\arcsin(\frac{a}{|OA|}) - |\theta|}{|\theta|} \to \infty. 
$$
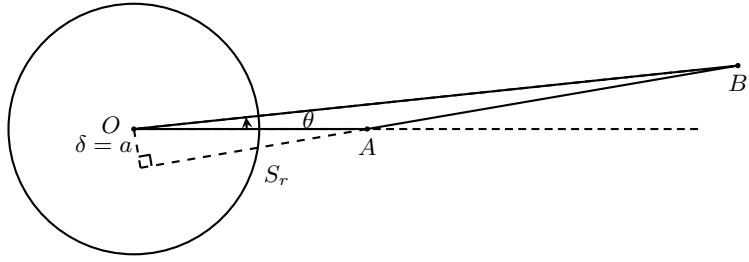
\begin{figure}
    \centering
    \begin{tikzpicture}[scale=5.5, line join=round, font=\normalsize]
        % 定义圆心 O (原点)
        \coordinate (O) at (0,0);
        
        % 绘制半径为 1/6 的圆 Γ
        \def\angle{6}
        \coordinate (E) at (0:1/3);
        \coordinate (F) at (\angle:1/3);
        \coordinate (E1) at (0:1.5);
        \coordinate (F1) at (\angle:1.5);
        \coordinate (A) at (0.620376788539, 0);
        \coordinate (B) at (1.60602706963,0.16880024697);
        \coordinate (P) at (-80.281929788:pi/30);
        \coordinate (H) at (-90+17.1481610717:0.0982812104998);
       
        % \draw[|<->|, >=stealth, thick] ($(O)+({\angle+90}:0.04)$) -- node[fill=white] {$R_1$} ($(F1)+({\angle+90}:0.04)$);
         % 绘制一个四分之一圆（90°圆弧）
        % \draw[thick] (E1) arc (0:\angle:1.5) -- cycle;
        
        % \node at ($(0.59,0.3)+(-55:0.05)$) {$1$};
        % \draw[|<->|, >=stealth, thick] ($(A)+(-54.7897182901:0.04)$) -- node[fill=white] {$1$} ($(B)+(-54.7897182901:0.04)$);
        \draw[thick] (O) circle (1/3);
        \draw[black, thick] (A) -- (O) -- (B) -- cycle;
        \draw[black, thick, densely dashed] (E1) -- (O) -- (F1);
        \draw[dashed,thick] (O)--(P);
        \draw[dashed,thick] (A)--(P);
        \pic [draw, thick, ->, >=stealth,angle radius=22mm, angle eccentricity=1.2, 
                  "{$\theta$}"] {angle = E--O--F};
        \pic [draw, thick, angle radius=1.5mm, angle eccentricity=1.55, 
                  "{}", font=\large] {right angle = O--P--A};
        \fill (O) circle (0.2pt) node[shift={(-0.3,0.05)}] {$O$};
        \fill (A) circle (0.2pt) ;
        \fill (B) circle (0.2pt) ;
        \node[shift={(-0.4,-0.2)}] at (O) {$\delta=a$};
        \node at (-18:0.40) {$S_{r}$};
        \node[below] at (A) {$A$};
        \node[below] at (B) {$B$};
    \end{tikzpicture}
    \caption{The case $l_\alpha\cap B_r = \varnothing$}
    \label{fig:l-out-Br}
\end{figure}
Therefore, a bounded range for the needle $ l_\alpha $, i.e., $ l_\alpha \subset B_{R_1} $, is necessary to ensure a positive lower bound for the arc length. 

Now we introduce a parameter $R_1>1$ to distinguish two cases regarding the location of the needle $l_\alpha$. Our next goal is to find a suitable choice of $ R_1 $.

Consider $ |OB| $ in Figure \ref{fig:l-alpha-in-Br}. We have $ |OB| \sin|\theta| = \frac{h}{|OA|} $ and $ \arcsin\frac{h}{|OA|} = \arcsin\frac{\delta_0}{r} $, where $h$ satisfies (\ref{eq:h}). It follows that $ \frac{\delta_0}{r} = |OB| \sin|\theta| $. Solving for $ |OB| $ yields:
\begin{equation}
    |OB|(\delta_0, r) = \frac{\sqrt{4\delta_0^2 + 1}}{1 - 2\sqrt{r^2 - \delta_0^2}}.
\end{equation}
% The function $ |OB| $ is decreasing in $ \delta_0 $ and increasing in $ r $. Denote $\lambda\in[0,1]$ and a num $r_\lambda = \lambda a+(1-\lambda)r_0$ between $a$ and $r_0$. Therefore, $ |OB|(\delta_0, r) \geq |OB|(\delta_1(r_\lambda), r_\lambda) $ for $r_\lambda \leq r \leq r_0$. Let $ R_1 = |OB|(\delta_1(r_\lambda), r_\lambda) $. Thus, for any $\beta\in J_\Gamma$, if $ l_\beta \subset B_{R_1} $, it follows that $ \beta \leq \alpha_1 $ for any $r$, $r_\lambda \leq r \leq r_0$. When $a \leq r \leq r_\lambda$, a direction $ \beta $ such that $ l_\beta \cap B_r = \varnothing $ and $ \Delta_\beta \cap B_r = I $, will satisfty $\beta\leq\arcsin(R_1\sin|\theta|)$. $J_\Gamma\subset[\alpha_1,\alpha_2]$ if $\alpha_1\leq\arcsin(R_1\sin\theta$, otherwise $J_\Gamma\subset[\theta-\arcsin(R_1\sin\theta),\arcsin(R_1\sin\theta)]$.
The function \(|OB|\) is decreasing in \(\delta_0\) and increasing in \(r\). Let \(\lambda \in [0,1]\) and define 
\[
    r_\lambda = \lambda a + (1-\lambda)r_0,
\] 
which lies between \(a\) and \(r_0\). Therefore, 
$$
|OB|(\delta_0, r) \geq |OB|(\delta_1(r_\lambda), r_\lambda)
$$ for \(r_\lambda \leq r \leq r_0\). Let 
\begin{equation}
    R_1 = |OB|(\delta_1(r_\lambda), r_\lambda).
\end{equation}
For convenience, we will use the same notation $J_\Gamma$ to denote 
$$
J_{\Gamma} = \{\alpha \mid \Delta_\alpha \subset B(0,R_1),  \Gamma_{\alpha,1} \subset \Gamma\}.
$$ 
Then, for any \(\beta \in J_\Gamma\), since \(l_\beta \subset B_{R_1}\), it follows that \(\beta \leq \alpha_1\) for any \(r \in [r_\lambda, r_0]\). When \(a \leq r \leq r_\lambda\), every direction \(\beta\) such that \(l_\beta \cap B_r = \varnothing\) and \(\Delta_\beta \cap B_r = \Gamma\) will satisfy \(\beta \leq \arcsin(R_1 \sin|\theta|)\). Since the point $A$ lies on the ray $OE$, similarly we have $\beta\geq|\theta|-\arcsin(R_1\sin|\theta|)$. Thus \(J_\Gamma \subset [\alpha_2, \alpha_1]\) if \(\alpha_1 \geq \arcsin(R_1 \sin|\theta|)\); otherwise, \(J_\Gamma \subset [|\theta| - \arcsin(R_1 \sin|\theta|), \arcsin(R_1 \sin|\theta|)]\).

Summarizing, we have proven the following. 
\begin{lemma}
    If $\Delta_\alpha \subset B(0,R_1)$, then $J_{\Gamma_{\alpha,1}} \subset [\alpha_2,\alpha_1]$, where $\alpha_1,\alpha_2$ correspond to the isosceles triangles $\Delta_{\alpha_1},\Delta_{\alpha_2}$ ($\alpha_2<\alpha_1$) having $\Gamma_{\alpha, 1}$ as an arc, for $r_\lambda\leq r\leq r_0$. 
\end{lemma}
% \begin{proof}
%     By we state above, for any $\beta\in J_{I_\alpha,1}$, $l_\beta\subset B_{R_1}$ follows that $\alpha_2\leq\beta\leq\alpha_1$.
% \end{proof}
\begin{lemma}
     Let $\Gamma$ be an arc in $S_r$ with central angle $\theta$ (see Figure \ref{fig:l-alpha-in-Br}), then 
     \begin{equation}
         |J_\Gamma| \leq g(r)|\theta|,
     \end{equation}
where 
\begin{equation}
g(r)=\max\left\{\frac{1+2r}{1-2r},\ \frac{1+2r_\lambda}{1-2r_\lambda},\ \frac{\pi}{ \frac{\pi}{2} - \arctan(2r)}\right\}.
\end{equation}
    \label{lem:control} 
\end{lemma}
\begin{proof}
    Let \(\delta_0\) be the distance from \(O\) to the needle $l_{\alpha,1}$. We consider two cases.\\
    \textbf{Case 1: }$ |\theta| \leq \frac{\pi}{2} - \arctan(2r) $. In this case, \(|\theta| = \arcsin(\delta_0 / r) - \arctan(2\delta_0)\) and \(|\alpha_1- \alpha_2| = 2\arcsin(\delta_0 / r) - |\theta|\). Their ratio satisfies:
    \[
    \frac{|J_\Gamma|}{|\theta|} \leq \frac{|\alpha_1- \alpha_2|}{|\theta|} \leq \frac{1+2r}{1-2r},
    \]
    since
    \begin{equation*}
        \frac{2\arcsin(\delta_0 / r)}{\arcsin(\delta_0 / r) - \arctan(2\delta_0)} \leq \frac{2}{1 - 2r}.
    \end{equation*}
    For \(a \leq r \leq r_\lambda\),
    \[
    \frac{|J_\Gamma|}{|\theta|} \leq \max \left\{ \frac{|\alpha_1-\alpha_2|}{|\theta|}, \frac{|2\arcsin(R_1 \sin|\theta|) - |\theta||}{|\theta|} \right\} \leq \frac{1+2r_\lambda}{1-2r_\lambda}.
    \]
    since 
    \begin{equation*}
        \frac{|\alpha_1-\alpha_2|}{|\theta|}\leq\frac{1+2r}{1-2r}\le\frac{1+2r_\lambda}{1-2r_\lambda},
    \end{equation*} 
    and
    \begin{equation*}
         \frac{|2\arcsin(R_1 \sin|\theta|) - |\theta||}{|\theta|}\leq \frac{|2\arcsin(\delta_0 / r_\lambda) - |\theta||}{|\theta|}\leq\frac{1+2r_\lambda}{1-2r_\lambda}.
    \end{equation*}
    \textbf{Case 2: }$ |\theta| \geq \frac{\pi}{2} - \arctan(2r) $. In this case, we have
    $$
    \frac{|J_\Gamma|}{|\theta|}\leq\frac{\pi}{ \frac{\pi}{2} - \arctan(2r)}.
    $$
    Combining these two cases, we complete the proof.
\end{proof}

Next, we introduce a parameter $ p \in [0, 1] $ to partition the range of directions.
 
\textbf{Case 2.1:} ${\mathcal{L}_1^*(\{\alpha\mid\Delta_\alpha \subset B_{R_1}\})\geq p\pi}.$
% Proportion $p$ of directions satisfy $\Delta_\alpha \subset B_{R_1}$.
% For such $\alpha\in\{\alpha\mid\Delta_\alpha \subset B_{R_1}\}$, let $\theta = |\Gamma_{\alpha,1}|/r$. 

Define $\mathscr{A}=\{\alpha \mid \Delta_\alpha\subset B_{R_1}\}$. Then we have
% \begin{align*}
%     \bigcup_{\alpha\in\mathscr{A}} J_{\Gamma_{\alpha,1}} &\subset \bigcup_{\alpha\in\mathscr{A}} \frac{g(r)}{r} \Gamma_{\alpha,1}
%     && \text{(by Lemma \ref{lem:control})} \\
%     &\subset \bigcup_{\alpha,i} \frac{g(r)}{r} \Gamma_{\alpha,i}
%     && (\text{where } (\alpha,i)\in[0,\pi)\times\{1,2\}) \\ 
%     &\subset \bigcup_k  \frac{g(r)}{r} \left( \frac{2}{1-\varepsilon} + 1 \right) \widetilde{\Gamma}_k
%     && (\text{by } (\ref{eq:vitalitype})).
% \end{align*}
\begin{align*}
    \bigcup_{\alpha\in\mathscr{A}} J_{\Gamma_{\alpha,1}} &\subset \bigcup_{\alpha\in\mathscr{A}} g(r) \theta_{\alpha,1}
    && \text{(by Lemma \ref{lem:control})} \\
    &\subset \bigcup_{\alpha,i} g(r) \theta_{\alpha,i}
    && (\text{where } (\alpha,i)\in[0,\pi)\times\{1,2\}).
    % &\subset \bigcup_k  \frac{g(r)}{r} \left( \frac{2}{1-\varepsilon} + 1 \right) \widetilde{\Gamma}_k
    % && (\text{by } (\ref{eq:vitalitype})).
\end{align*}
The arc length is bounded by:
\begin{align*}
    p\pi &\leq \mathcal{L}_1^* \left(\bigcup_{\alpha\in\mathscr{A}} J_{\Gamma_{\alpha,1}}\right) 
    && (\text{since } \alpha \in J_{\Gamma_{\alpha,1}}) \\
    &\leq \mathcal{L}_1^* \left(\bigcup_{\alpha,i} g(r) \theta_{\alpha,i}\right) 
    && \\
    % &=\frac{1}{r}\mathcal{H}_1^*\left(\bigcup_{\alpha,i}g(r)\Gamma_{\alpha,i}\right)
    % \\
    % &\leq \left| \bigcup_k  \frac{g(r)}{r} \left( \frac{2}{1-\varepsilon} + 1 \right) \widetilde{\Gamma}_k \right| \\
    &\leq \left| \bigcup_k  g(r) \left( \frac{2}{1-\varepsilon} + 1 \right) \widetilde{\theta}_k \right|
    &&  (\text{by } (\ref{eq:thetavitali}))\\
    &\leq \sum_k g(r) \left( \frac{2}{1-\varepsilon} + 1 \right) |\widetilde{\theta}_k|
    && \text{(by the subadditivity of measure)}.
\end{align*}
As $\widetilde{\Gamma}_k$ are disjoint, we have $\mathcal{H}_1^* (E\cap S_r) \geq \sum |\widetilde{\Gamma}_k|=\sum r\cdot|\widetilde{\theta}_k|$. Taking $\varepsilon \to 0$, we obtain:
\begin{equation}
    \mathcal{H}_1^*(E \cap S_r) \geq p\frac{\pi}{3} \cdot \frac{1}{g(r)}r.
\end{equation}
% Denote  
% \begin{equation}
%     g(r) = \frac{\pi}{3}\frac{1-2r}{1+2r}r = \frac{\pi}{3}(1-r-\frac{1}{2r+1}). 
% \end{equation}

Thus, by Lemma \ref{lem:inplusout}, the outer measure of $E$ satisfies:
\begin{equation}
\label{eq:caseIbound}
\mathcal{L}_2^*(E) \geq p\frac{\pi}{3} \left(1 - \frac{f(r_0)}{2r_0^2}\right) \left(\int_a^{r_0} \frac{1}{g(r)} r\,dr\right) + \frac{\pi}{4} f(r_0),  
\end{equation}
where $f(r) = \frac{1}{2}r(2r-1)^2$.

\textbf{Case 2.2:} ${\mathcal{L}_1^*(\{\alpha\mid\Delta_\alpha\subset B_{R_1}\})\leq p\pi}.$
% Proportion $1-p$ of directions satisfy $\Delta_\alpha \not\subset B_{R_1}$.

In this case $\mathcal{L}_1^*(\{\alpha\mid\Delta_\alpha \not\subset B_{R_1}\})\geq \mathcal{L}_1^*([0,\pi))-\mathcal{L}_1^*(\{\alpha\mid\Delta_\alpha\subset B_{R_1}\})\geq(1-p)\pi$. 
Thus, a positive proportion of directions corresponds to triangles lying outside the disk $B_{R_1}$, which contribute a positive area. To estimate this contribution, we now establish a disjointness property for the interior parts of these triangles.
\begin{lemma}
    \label{the:outofcircle}
    If two needles in directions $\alpha_1,\alpha_2$ lie outside $S_r$, and $|\alpha_1 - \alpha_2| \geq \arcsin (\delta_1 / r)  + \arcsin (\delta_2 / r)$, then $\mathring\Delta_1^{\text{\rm int}} \cap \mathring\Delta_2^{\text{\rm int}} = \varnothing$. 
\end{lemma}

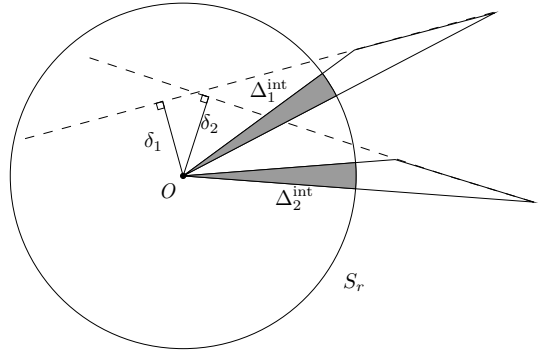
\begin{figure}[!ht]
    \centering
    % TikZ diagram code here (unchanged)
     \scalebox{0.8}{
    \begin{tikzpicture}[x=0.75pt,y=0.75pt,yscale=-1,xscale=1,font=\Large]
            %uncomment if require: \path (0,380); %set diagram left start at 0, and has height of 380
            
            %Shape: Circle [id:dp8047139529409595] 
            \draw   (233,173) .. controls (233,112.8) and (281.8,64) .. (342,64) .. controls (402.2,64) and (451,112.8) .. (451,173) .. controls (451,233.2) and (402.2,282) .. (342,282) .. controls (281.8,282) and (233,233.2) .. (233,173) -- cycle ;
            %Shape: Circle [id:dp5666051308176348] 
            \draw  [fill={rgb, 255:red, 0; green, 0; blue, 0 }  ,fill opacity=1 ] (340.34,173) .. controls (340.34,172.09) and (341.08,171.34) .. (342,171.34) .. controls (342.91,171.34) and (343.66,172.08) .. (343.66,173) .. controls (343.66,173.91) and (342.92,174.66) .. (342,174.66) .. controls (341.09,174.66) and (340.34,173.92) .. (340.34,173) -- cycle ;
            %Shape: Triangle [id:dp4847175371100867] 
            \draw   (449.99,93.59) -- (538.34,70.18) -- (342,173) -- cycle ;
            %Shape: Triangle [id:dp7257370586941199] 
            \draw   (475.65,162.7) -- (563.02,189.54) -- (342,173) -- cycle ;
            %Shape: Boxed Line [id:dp09311160579659306] 
            \draw  [dash pattern={on 4.5pt off 4.5pt}]  (283,98.5) -- (506.55,172.3) -- (524.34,177.73) -- (563.02,189.54) ;
            %Shape: Boxed Line [id:dp04801483033145093] 
            \draw  [dash pattern={on 4.5pt off 4.5pt}]  (242.33,149.5) -- (361.53,117.35) -- (421.13,101.28) -- (480.73,85.2) -- (540.32,69.13) ;
            %Shape: Right Angle [id:dp4929027190728523] 
            \draw   (327.76,126.34) -- (329.07,125.98) -- (342,173) ;
            %Shape: Square [id:dp3878188907835374] 
            \draw   (329.07,125.98) -- (330.25,130.06) -- (326.17,131.25) -- (324.99,127.17) -- cycle ;
            %Shape: Path Data [id:dp5038854856481851] 
            \draw  [fill={rgb, 255:red, 155; green, 155; blue, 155 }  ,fill opacity=1 ] (342,173) -- (429.82,108.42) .. controls (433.08,112.85) and (436.01,117.52) .. (438.58,122.42) -- (342,173) -- cycle ;
            %Shape: Path Data [id:dp039884627024831865] 
            \draw  [fill={rgb, 255:red, 155; green, 155; blue, 155 }  ,fill opacity=1 ] (451,173) .. controls (451,175.74) and (450.9,178.45) .. (450.7,181.14) -- (342,173) -- (450.68,164.62) .. controls (450.89,167.39) and (451,170.18) .. (451,173) -- cycle ;
            %Shape: Right Angle [id:dp12721358334246746] 
            \draw   (353.71,121.79) -- (358.08,123.2) -- (342,173) ;
            %Shape: Square [id:dp39151346916823715] 
            \draw   (354.21,121.95) -- (358.08,123.2) -- (356.83,127.07) -- (352.96,125.82) -- cycle ;
            
            % Text Node
            \draw (325.33,176.68) node [anchor=north west][inner sep=0.75pt]   [align=left] {\textit{O}};
            % Text Node
            \draw (316.17,144.5) node [anchor=north west][inner sep=0.75pt]   [align=left] {$\displaystyle \delta _{1}$};
            % Text Node
            \draw (351.6,133.13) node [anchor=north west][inner sep=0.75pt]   [align=left] {$\displaystyle \delta _{2}$};
            % Text Node
            \draw (398.67,178.33) node [anchor=north west][inner sep=0.75pt]   [align=left] {$\displaystyle \Delta^{\text{\rm int}}_{2}$};
            % Text Node
            \draw (382.67,110.33) node [anchor=north west][inner sep=0.75pt]   [align=left] {$\displaystyle \Delta ^{\text{\rm int}}_{1}$};
            \node at (450,240) {$S_{r}$};
        \end{tikzpicture}
        }
    \caption{$\Delta_1^{\text{\rm int}}$ and $\Delta_2^{\text{\rm int}}$ are disjoint for sufficiently separated directions}
    \label{fig:outofcircle}
\end{figure}

\begin{proof}
    Suppose $P \in \mathring\Delta_1^{\text{\rm int}} \cap \mathring\Delta_2^{\text{\rm int}}$ (see Figure \ref{fig:outofcircle}). The ray $OP$ intersects the needles at $K_1,K_2$ and intersect $S_r$ at $Q$. Assume $OK_1 \subset OK_2$. Then $OK_1 \subset \mathring\Delta_1 \cap \mathring\Delta_2$, so $QK_1 \subset \mathring\Delta_1^{\text{\rm ext}} \cap \mathring\Delta_2^{\text{\rm ext}}$, contradicting $|\alpha_1 - \alpha_2| \geq \arcsin(\delta_1/r) + \arcsin(\delta_2/r)$.
\end{proof}

The needles which are not contained in $B(0,R_1)$ are disjoint from $B(0,R_1 - 1)$. Analogous to Cunningham's proof, we have the following (the proof is by simple calculus, so we omit it here):
\begin{lemma}
    \label{the:min}
    For $x \in [0,a]$, $\frac{x}{2 \arcsin(x/r)} \geq c(r)$, where $c(r) = \frac{a}{2 \arcsin(a/r)}$.
\end{lemma}
With Lemma 3.4, we can now prove
\begin{theorem}
Let $A = \{\alpha \mid l_\alpha \subset B(0,r)^c\}$. Then we have
\[
\mathcal{L}_2^*(\bigcup_{\alpha\in A}\Delta_\alpha) \geq \frac{\mathcal{L}_1^*(A)}{4} c(r),
\]
\text{where} $c(r)$ \text{is as in Lemma \ref{the:min}}.
\label{the:outcirarea}
\end{theorem}
\begin{proof}
The proof is similar to that of Theorem \ref{the:phiareamin}. So we summarize the distinctions below:
\begin{enumerate}
    \item Now, take $A = \{\alpha \mid l_\alpha \subset B(0,r)^c\}$.
    \item The inequality chain (\ref{eq:geqchain}) now becomes
    \begin{equation*}
        \mathcal{L}_2^*(E) \geq \sum_{n=1}^k |\Delta_n|\geq\sum_{n=1}^kc(r)\arcsin(\delta_n/r)\geq \frac{\mathcal{L}_1^*(A)}{4}c(r),
    \end{equation*}
    where we have applied Lemma \ref{the:min} in the second inequality.
    \item In Case 2 of the proof of Theorem \ref{the:phiareamin}, the lower bound (\ref{eq:linemeasure}) now becomes 
    $$
     \mathcal{L}_2^*\left((\sqcup_k\mathring\Delta_k)^c\cap\bigcup_{\alpha\in A_0}l_\alpha\right)\geq\frac{1}{2} r^2 \left( \mathcal{L}_1^*(A) - \frac{4}{f(r)} b \right),
    $$
    and $B_r$ below \eqref{eq:linemeasure} is replaced by $(\sqcup_k\mathring\Delta_k)^c$. 
    \item Throughout the proof, $f(r)$ is replaced by $c(r)$.
\end{enumerate}
The rest of the proof is identical to that of Theorem \ref{the:phiareamin}.
\end{proof}
By Theorem \ref{the:outcirarea}, we obtain the following lower bound in Case II:
\begin{equation}
    \label{eq:caseIIbound}
    \mathcal{L}_2^*(E)\geq \frac{(1-p)\pi}{4}c(R_1-1).
\end{equation}

Now we are ready to prove Theorem \ref{the:mylowerbound}. 
% \begin{theorem}
% \label{the:mylowerbound}
% Every star-shaped Kakeya set $E$ satisfies
% \[
% \mathcal{L}_2^*(E) \geq \frac{\pi}{98}.
% \]
% \end{theorem}
\begin{proof}[Proof of Theorem \ref{the:mylowerbound}]
    Combining Cases I and II above, we have
    \[
    \mathcal{L}_2^*(E) \geq \min \left\{p\frac{\pi}{3} \left(1 - \frac{f(r_0)}{2r_0^2}\right) \left(\int_a^{r_0} \frac{1}{g(r)} r\,dr\right) + \frac{\pi}{4} f(r_0),  \frac{\pi(1-p)}{4} c(R_1-1) \right\}.
    \]
    Now choose $p = 9/10$, $\lambda = 9/10$, which nearly balance the two terms above (see §\ref{subsec:optimization} below for the optimal choice of $p$ and $\lambda$). We get 
    $$
    \mathcal{L}_2^*(E) \geq \min \{ (0.010205\cdots)\pi, (0.0107\cdots)\pi \} = (0.010205\cdots)\pi > \pi/98.
    $$ Thus, we obtain:
    \[
    \mathcal{L}_2^*(E) \geq \frac{\pi}{98}.
    \]
    This completes the proof of Theorem \ref{the:mylowerbound}.
\end{proof}

\section{Remarks}
\subsection{Further refinements}
\label{subsec:optimization}

The choice of parameters described above is not optimal. One can solve a constrained optimization problem to achieve a better bound. More precisely, assuming a universal lower bound $\mathcal{L}_2^*(E) \ge \frac{1}{2}a$ (where $a$ is the parameter chosen to be $\frac{\pi}{49}$ above), we may restrict our attention to the directions $\alpha$ for which the distance $\delta$ from $O$ to the needle $l_\alpha$ satisfies $\delta < a$. Let $r_0$ be the radius of the cut-off circle and $p$ the proportion of the directions as in the proof of Lemma \ref{lem:control}. Now choose $p$ such that the expressions from \eqref{eq:caseIbound} and \eqref{eq:caseIIbound} balance:
\[
p\frac{\pi}{3}\Bigl(1-\frac{f(r_0)}{2r_0^2}\Bigr)\Bigl(\int_a^{r_0}\frac{r}{g(r)}\,dr\Bigr)+\frac{\pi}{4}f(r_0)=\frac{\pi(1-p)}{4}c(R_1-1).
\]
Then choose $a$ and $r$ to enlarge this minimum under the constraint
\[
\frac{\pi(1-p)}{4}c\ge\frac12 a .
\]
Finally, taking $a\in[0.06473,0.06474]$, $r_0\in[0.22785,0.22786]$, $p\in[0.88794,0.88795]$, and $\lambda\in[0.90696,0.90697]$ gives an improved lower bound:
\begin{equation}
    \mathcal{L}_2^*(E) \ge (0.01030\cdots)\pi .
    \label{eq:largerbound}
\end{equation}

% We also know that even the lower bound obtained in this way is not optimal.
We remark also that, even the lower bound obtained in such a way is not optimal. Indeed, a better bound can be obtained iteratively, as follows: By rescaling the configuration from $B_{R_1}$ to $B_a$, one has
\[
\mathcal{L}_2^*(E\cap B_a)\ge\Bigl(\frac{a}{R_1}\Bigr)^2\cdot\Bigl[\, p\frac{\pi}{3}\Bigl(1-\frac{f(r_0)}{2r_0^2}\Bigr)\Bigl(\int_a^{r_0}\frac{r}{g(r)}\,dr\Bigr)+\frac{\pi}{4}f(r_0)\Bigr].
\]
Adding the contribution from the annulus $B_{r_0}\setminus B_a$;
\[
\mathcal{L}_2^*(E\cap (B_{r_0}\setminus B_a))\ge p\frac{\pi}{3}\int_a^{r_0}\frac{r}{g(r)}\,dr,
\]
we obtain a larger lower bound for $\mathcal{L}_2^*(E\cap B_{r_0})$ than the one obtained by considering the annulus alone. Plugging this improved lower bound for the inner part into Lemma~\ref{lem:inplusout} then yields a better global bound. Repeating this argument yields the lower bound
\begin{equation}
    \mathcal{L}_2^*(E)\ge\frac{p\frac{\pi}{3}\bigl(1-\frac{f(r_0)}{2r_0^2}\bigr)\bigl(\int_a^{r_0}\frac{r}{g(r)}\,dr\bigr)+\frac{\pi}{4}f(r_0)}{1-\bigl(1-\frac{f(r_0)}{2r_0^2}\bigr)\bigl(\frac{a}{R_1}\bigr)^2},
    \label{eq:iterativecaseIbound}
\end{equation}
which is exactly \eqref{eq:caseIbound} multiplied by $\displaystyle\frac{1}{1-\bigl(1-\frac{f(r_0)}{2r_0^2}\bigr)\bigl(\frac{a}{R_1}\bigr)^2}$. Finally, taking\\$a\in[0.06600,0.06601]$, $r_0\in[0.24072,0.24073]$, $p\in[0.87454,0.87455]$, and $\lambda\in[0.16679,0.16680]$ gives a larger lower bound:
\begin{equation}
    \mathcal{L}_2^*(E)\ge (0.01050\cdots)\pi .
    \label{eq:iterativelowerbound}
\end{equation}
\subsection{A double integral approach}
\label{subsec:integral-approach}

Another possible approach to improve the lower bound is to analyze the following integral:
\begin{equation*}
\int_0^\infty \mathcal{H}_1^*\left( E\cap S_r \right) dr
\end{equation*}
using Fubini's theorem. More precisely, one may select a disjoint collection of arcs $\widetilde{\Gamma}_k$ from $\{\Delta_\alpha \cap S_r\mid\alpha\in[0,\pi)\}$. For each such arc, consider the angle $\Delta\theta_k = \angle OAB$, where $A$ and $B$ are the endpoints of the needle corresponding to the arc. By relating the arc length $|\widetilde{\Gamma}_k|$ to the angle $\Delta\lambda_k$ via a function $f$ such that $|\widetilde{\Gamma}_k| = f(\Delta_{\alpha_k},r) \cdot |\Delta\lambda_k|$, and applying Fubini's theorem, it would then suffice to establish a lower bound for $\int_0^\infty f(\Delta_\alpha,r)\, dr$, which can be further analyzed by minimizing the function over its domain.

\subsection{Related problems}
\label{subsec:open-problems}
\subsubsection*{Problem 1}

It is unknown whether there exists a star-shaped Kakeya set \(E\) with area smaller than the one given by Cunningham and Schoenberg in~\cite{Cunningham_Schoenberg_1965}, i.e., with \(\mathcal{L}_2^*(E) < \frac{(5 - 2\sqrt{2})\pi}{24}\). Moreover, the construction in~\cite{Cunningham_Schoenberg_1965} allows continuous rotation of the needle; even under this stronger condition, it remains unclear whether their construction is optimal. 

\subsubsection*{Problem 2}
To the best of our knowledge, the star-shaped Kakeya problem in dimensions three and higher has not been studied before. However, continuously parametrized Kakeya sets in $\mathbb{R}^n$ have been considered, for example, in~\cite{Järvenpää_Järvenpää_Keleti_Máthé_2011} and ~\cite{FuGan2023}. A recent breakthrough by Wang and Zahl ~\cite{WangZahl2025} on the Kakeya conjecture implies that a star-shaped Kakeya set in $\mathbb{R}^3$ would at least have Hausdorff dimension 3. 

\appendix
\section{Technical Lemmas}
This appendix includes two technical lemmas used in Sections \ref{sec:2} and \ref{sec:3}. 
% Their proofs are standard and are omitted for brevity.
\begin{lemma}
    Let \( A \subset [0, 2\pi) \), and let \( B = [0, r] \). Define
\[
T = \{ (x,y) \in \mathbb{R}^2 : x = \rho \cos \theta, y = \rho \sin \theta, \theta \in A, \rho \in B \}.
\]
Then we have
\[
\mathcal{L}_2^*(T) = \frac{1}{2} r^2 \, \mathcal{L}_1^*(A).
\]
\label{lem:polmeasure}
\end{lemma}
\begin{proof}
First, we show  ${ \mathcal{L}_2^*(T) \leq \frac{1}{2} r^2 \mathcal{L}_1^*(A)}$.
For any \( \varepsilon > 0 \), there exists an open set \( C \supseteq A \) such that \( \mathcal{L}_1(C) < \mathcal{L}_1^*(A) + \varepsilon \). Let \( T_C \) be the image of \( C \times [0, r] \) under the polar map. Then \( T \subset T_C \), and \( T_C \) is measurable. Using polar integration:
\[
\mathcal{L}_2(T_C) = \int_C \int_0^r \rho \, d\rho \, d\theta = \frac{1}{2} r^2 \mathcal{L}_1(C) < \frac{1}{2} r^2 (\mathcal{L}_1^*(A) + \varepsilon).
\]
Hence, \( \mathcal{L}_2^*(T) \leq \frac{1}{2} r^2 \mathcal{L}_1^*(A) \).

Next, we show ${ \mathcal{L}_2^*(T) \geq \frac{1}{2} r^2 \mathcal{L}_1^*(A)}$.
Suppose \( \mathcal{L}_2^*(T) < \frac{1}{2} r^2 \mathcal{L}_1^*(A) \). Then there exists an open set \( V \supseteq T \) with \( \mathcal{L}_2(V) < \frac{1}{2} r^2 \mathcal{L}_1^*(A) \). Define \( f(\theta) = \int_0^\infty \chi_{V} \rho \, d\rho \). By Tonelli’s theorem, \( f \) is measurable. For \( \theta \in A \), \( f(\theta) \geq \frac{1}{2} r^2 \). Let \( E = \{ \theta : f(\theta) \geq \frac{1}{2} r^2 \} \), which is measurable and contains \( A \). Then:
\[
\mathcal{L}_2(V) = \int f(\theta) \, d\theta \geq \int_E \frac{1}{2} r^2 \, d\theta = \frac{1}{2} r^2 \mathcal{L}_1(E) \geq \frac{1}{2} r^2 \mathcal{L}_1^*(A),
\]
a contradiction. Thus, \( \mathcal{L}_2^*(T) \geq \frac{1}{2} r^2 \mathcal{L}_1^*(A) \).

Combining both bounds, we see that \( \mathcal{L}_2^*(T) = \frac{1}{2} r^2 \mathcal{L}_1^*(A) \).
\end{proof}
\begin{lemma}
    \label{lem:polar_fubini}
    For any set $E \subset \mathbb{R}^2$, we have
    \begin{equation}
    \mathcal{L}_2^*(E) \geq \int_{\mathbb{R}^+}^* \mathcal{H}_1^*(E \cap S_r) dr,
    \label{eq:appfubini}
    \end{equation}
    where $S_r$ denotes the circle of radius $r$ centered at the origin.
\end{lemma}
\begin{proof}
    The equation (\ref{eq:appfubini}) can be justified as follows: for any \( \varepsilon > 0 \), by the definition of outer measure, there exists an open set \( G \subset \mathbb{R}^2 \) such that \( E \subset G \) and \( \mathcal{L}_2(G) < \mathcal{L}_2^*(E) + \varepsilon \), where \( \mathcal{L}_2(G) \) is the Lebesgue measure of \( G \) (note that \( G \) is measurable). Since \( G \) is measurable, we may apply Fubini's theorem in polar coordinates. The measure of \( G \) can be expressed as:
\[
\mathcal{L}_2(G) = \int_0^\infty \mathcal{H}_1(G \cap S_r)\,  dr,
\]
where $\mathcal{H}_1(\cdot)$ on the right-hand side denotes the arc length on \( S_r \). For each \( r > 0 \), the inclusion \( E \cap S_r \subset G \cap S_r \) implies:
\[
\mathcal{H}_1^*(E \cap S_r) \leq \mathcal{H}_1(G \cap S_r).
\]
Integrating both sides with respect to \( r \) gives:
\[
\int_{\mathbb{R}^+}^* \mathcal{H}_1^*(E \cap S_r)\,  dr \leq \int_0^\infty \mathcal{H}_1(G \cap S_r) \, dr = \mathcal{L}_2(G) < \mathcal{L}_2^*(E) + \varepsilon.
\]
Since \( \varepsilon > 0 \) is arbitrary, it follows that:
\[
\int_{\mathbb{R}^+}^* \mathcal{H}_1^*(E \cap S_r) \, dr \leq \mathcal{L}_2^*(E),
\]
which is equivalent to the desired inequality.

\end{proof}

% \section*{Acknowledgments}
% The author is deeply grateful to Xianghong Chen for his invaluable assistance and insightful suggestions that greatly contributed to resolving the central problem of this work. The author also wishes to express sincere appreciation for his constructive comments on improving the manuscript. Research supported in part by the National Natural Science Foundation of China. 

\section*{Acknowledgments}
The author is grateful to Xianghong Chen for his invaluable suggestions and to Lixin Yan for his helpful advice during seminar presentations. Research supported in part by the National Natural Science Foundation of China (Grant Nos. 12371105, 12426204). The author is also grateful to the anonymous referee for their helpful comments which improved the presentation of this work. 
% \begin{thebibliography}{9}
% \bibliographystyle{amsplain} %(amsplain or amsalpha)
\bibliography{main.bib} 
\nocite{Blank1963}

\end{document}